\documentclass[10pt]{article}
\usepackage[letterpaper, margin=1.85in]{geometry}
\usepackage{graphicx}
\usepackage{amsmath,amssymb}
\usepackage{amsthm}
\usepackage{stmaryrd}  
\usepackage{acronym}   
\usepackage{enumitem}  
\usepackage{psfrag}
\usepackage[labelfont=bf]{caption} 
\usepackage{subcaption}
\usepackage[breaklinks]{hyperref}

\usepackage{breakurl}

\newif\ifuseboldmathops
\newif\ifuseittextabbrevs
\useboldmathopstrue   

\ifuseittextabbrevs
	\newcommand{\cf}{{\it cf.}}
	\newcommand{\eg}{{\it e.g.}}
	\newcommand{\ie}{{\it i.e.}}
	\newcommand{\etc}{{\it etc.}}
	
\else
	\newcommand{\cf}{cf.}
	\newcommand{\eg}{e.g.}
	\newcommand{\ie}{i.e.}
	\newcommand{\etc}{etc.}
	
\fi

\ifuseboldmathops
	\newcommand{\reals}{\mathbf{R}}

\else
	\newcommand{\reals}{\mathbb{R}}

\fi

\ifuseboldmathops

\else

\fi

\ifuseboldmathops

\else

\fi

\ifuseboldmathops

\else

\fi

\newcommand{\eqbydef}{\mathrel{\stackrel{\Delta}{=}}}

\newcommand{\ones}{\mathbf 1}

\font\tensc=cmcsc10
\font\sevensc=cmcsc10 at 7pt
\font\fivesc=cmcsc10 at 5pt

\def\mathsc#1{{\mathchoice
  {\hbox{\tensc#1}}
  {\hbox{\tensc#1}}
  {\hbox{\sevensc#1}}
  {\hbox{\fivesc#1}}
}}

\def\ttbox#1{\mbox{\tt #1}}

\newcommand{\sat}{\mathop{\rm sat}}
\newcommand{\dz}{\mathop{\rm dz}}
\newcommand{\nn}{\mbox{\sc nn}}
\newcommand{\amn}{\varphi}

\newcommand{\Var}{\mathop{\rm Var}}
\newcommand{\Input}{\mathop{\rm In}}
\newcommand{\Output}{\mathop{\rm Out}}
\newcommand{\Graph}[1]{\llbracket #1 \rrbracket}

\newcommand{\smt}{\mbox{\sc smt}}
\newcommand{\mip}{\mbox{\sc mip}}

\newcommand{\card}{\mathop{\bf card}}

\newcommand{\amnet}{\textsc{Amnet}}

\theoremstyle{remark}

\newtheorem{example}{Example}

\theoremstyle{definition}
\newtheorem{definition}{Definition}

\theoremstyle{plain}

\newtheorem{theorem}{Theorem}

\acrodef{lmi}[LMI]{Linear Matrix Inequality}
\acrodef{mn}[$\mu$N]{$\mu$-network}
\acrodef{amn}[AMN]{affine multiplexing network}
\acrodef{iten}[ITEN]{if-then-else network}
\acrodef{aiten}[AITEN]{affine if-then-else network}
\acrodef{smt}[SMT]{Satisfiability Modulo Theory}
\acrodef{mip}[MIP]{Mixed Integer Programming}

\title{\bf Affine Multiplexing Networks:\\
		System Analysis, Learning, and Computation}

\author{Ivan Papusha \quad Ufuk Topcu \quad Steven Carr \quad Niklas Lauffer}
\date{April 30, 2018}

\begin{document}
\maketitle

\begin{abstract}
We introduce a novel architecture and computational framework for formal,
automated analysis of systems with a broad set of nonlinearities in the
feedback loop, such as neural networks, vision controllers, switched systems,
and even simple programs.  We call this computational structure an \emph{affine
multiplexing network} (AMN).  The architecture is based on interconnections of
two basic conceptual building
blocks: multiplexers ($\mu$), and affine transformations ($\alpha$). 
When attached together appropriately, these building blocks translate
to conjunctions and disjunctions of affine statements, resulting in an encoding
of the network into satisfiability modulo theory (SMT), mixed integer
programming, and sequential convex optimization solvers.

We show how to formulate and verify system properties like stability and
robustness, how to compute margins, and how to verify performance through a
sequence of SMT queries. As illustration, we use the framework to verify
closed loop, possibly nonlinear dynamical systems that contain neural networks
in the loop, and hint at a number of extensions that can make AMNs a potent
playground for interfacing between machine learning, control, convex and
nonconvex optimization, and formal methods. 
\end{abstract}

\section{Introduction}\label{sec:introduction}

\subsection{Affine multiplexing networks}
This work proposes a novel computational structure called an
\acfi{amn}, or an \emph{affine if-then-else network}, which is formed by the
composition of multiplexing functions and affine transformations in a dimension
compatible way, and parameterized by the weights and biases of the affine
transformations.  

By repeatedly instantiating and connecting these two components in an acyclic
computation graph, we can construct arbitrary relations between the
entire network's inputs and outputs.  The result of such an interconnection is
an artificial neural network with interspersed multiplexing
nonlinearities. Models built with the \acs{amn} viewpoint are much more
powerful in practice for certain applications than general neural network models, because
as we will see, many mathematical properties of AMNs can be formally and
automatically verified with existing tools. Meanwhile, AMNs can be trained just
as easily as neural networks.

The \acs{amn} model encompasses many other neural networks as a special case,
including deep multilayer feedforward networks with
piecewise linear nonlinearities, \eg,~rectified linear unit (ReLU), absolute
value, saturation, deadzone, and max-pooling nonlinearities. In particular,
\acsp{amn} can be readily applied to quantifying 
resilience in classifiers~\cite{Cheng:2017}, in software
safety verification~\cite{Huang:2017,Ehlers:2017}, and in
certification~\cite{Katz:2017}.

\paragraph{Building blocks}
Define the \emph{multiplexing} function $\mu:\reals^n \times \reals^n \times
\reals \to \reals^n$, 
\begin{equation}
	\label{eq:mu}
	\mu(x,y,z) \eqbydef \left\{
		\begin{array}{ll}
			x, & \text{if}~z\leq 0,\\
			y, & \text{otherwise}.
		\end{array}
	\right.
\end{equation}
The function $\mu$ represents a ternary \emph{choice} assignment, similar to the
operation
\[
	w := \mu(x,y,z) 
	\quad\Longleftrightarrow\quad
	w := \ttbox{if}\; z\leq 0\; \ttbox{then}\; x\; \ttbox{else}\; y.
\]
This multiplexing function is the first building block of an AMN, and 
can be visualized as a $2$-to-$1$ multiplexing unit, shown
in the left pane of
Figure~\ref{fig:basicblocks}. The value of $\mu(x,y,z)$ is either
$x$ or $y$, depending on whether the statement $z \leq 0$ is true ($1$)
or false ($0$).  Motivated by electronic component nomenclature, we refer
to $x$ and $y$ as the \emph{signals} or \emph{inputs}, and $z$ as the
\emph{select} or \emph{enable} input. If the value of $z$ satisfies the
one-dimensional linear \emph{enable condition} ($z\leq 0$), then the
\emph{output} $w$ obtains the value $x$; otherwise ($z > 0$) the output $w$
obtains the value~$y$.

The second building block is an \emph{affine transformation}, which is any function,
$\alpha : \reals^n \to \reals^m$, with a vector input and output, of the form
\begin{equation}
	\label{eq:aff}
	\alpha(x) \eqbydef Ax + b, 
	\quad A \in \reals^{m\times n},
	\quad b\in\reals^m.
\end{equation}
An affine transformation is visualized as an amplifier, shown in the right
pane of Figure~\ref{fig:basicblocks}, and parameterized by a \emph{gain} or
\emph{weight} matrix $A$ and a \emph{bias} vector~$b$.

\begin{figure}[hbtp]
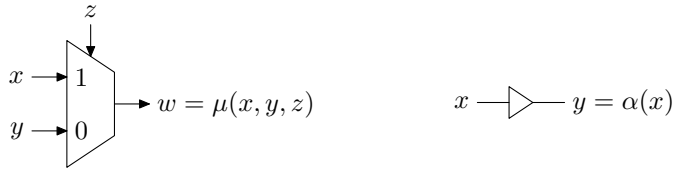

	\centering
    \begin{subfigure}{0.4\linewidth}
		\centering
        \includegraphics{fig/mp/mux-1}
		\label{fig:mux1}
    \end{subfigure}
    \quad
    \begin{subfigure}{0.4\linewidth}
		\centering
        \includegraphics[trim=0 0 0 -12pt]{fig/mp/aff-1}
        \label{fig:aff1}
    \end{subfigure}
	\caption{Basic building blocks: multiplexer ($\mu$) and affine
	transformation ($\alpha$).}
    \label{fig:basicblocks}
	\vspace{-5ex}
\end{figure}

\paragraph{Motivation}
An \acl{amn}'s motivating capabilities are ultimately realized by its encoding in linear
arithmetic.  This encoding allows one to ask and answer quantitative questions
about the network function, or any other function when expressed as an
\acs{amn}, using \ac{smt} or \ac{mip} solvers. 

This formal encoding feature of \acsp{amn} makes them particularly attractive
for the analysis and verification of \emph{control systems} with nonlinear
components in the loop, including switched systems, classifiers, 
and neural networks.
In the first part of this paper, \S\ref{sec:introduction}, we formally define
and explore examples of \acsp{amn}, and explain how to encode them using
\acs{smt} and \ac{mip}. 
Ultimately motivated by the modeling capabilities of \acsp{amn} in the loop
with control systems, we describe a powerful counterexample-guided
computational procedure to search for Lyapunov functions, which we will
describe in \S\ref{sec:inloop}. We give several extended examples in
\S\ref{sec:examples}, and conclude with a number of directions for future work
in \S\ref{sec:extensions}.

\subsection{Formal definition}
\begin{definition}[Affine Multiplexing Network]
An \emph{\acl{amn}} is a real vector function
$\amn:\reals^q\to\reals^p$ that can be expressed recursively as
\begin{equation}
	\label{eq:amn_bnf}
	\amn(x) ::= x \mid \alpha(\amn_1(x)) \mid \mu(\amn_1(x),\amn_2(x),\amn_3(x)),
\end{equation}
where $x$ is the (vector) input variable,
$\mu$ is the multiplexing function~\eqref{eq:mu}, 
$\alpha$ is any affine transformation of the form~\eqref{eq:aff}, 
and $\amn_1,\ldots,\amn_3$ are any other
\aclp{amn} with compatible dimensions.
\end{definition}

We distinguish between the network function $\amn$ (or $\amn[x]$, where the
input variable $x$ has not been bound to any particular value), and its
evaluation at an input, $\amn(a)=\amn[x:=a]$.  For a given assignment to the
input variable $x$, each terminal expression in~\eqref{eq:amn_bnf} evaluates to
a vector of appropriate dimension. The entire network $\amn$ is parameterized
by the
weights and biases of all its affine expressions, which we lump into a
single $r$-dimensional vector $\theta\in\reals^r$. When important, we write
$\amn_\theta(x)$ to stress that the function $\amn$ is parameterized by
$\theta$.

The input and output dimensions $q$ and $p$ can be arbitrary and
different for different \acs{amn} instances, as long
as the full recursive expression~\eqref{eq:amn_bnf} makes sense. Constants are
affine transformations independent of the input, \eg, $\alpha(x)=c$.  
As seen in the next section, many common functions can be rewritten in AMN form.

\subsection{Examples}
\begin{example}[Maximum]\label{ex:max}
The function $\max:\reals^2\to\reals$ ($q=2$, $p=1$) that computes the maximum
of two numbers can be expressed as an \acs{amn}.
Let $x=(x_1,x_2)\in\reals^2$ be the input variable, and define
affine transformations $\alpha_i:\reals^2\to\reals$, $i=1,\ldots,3$, by
\[
	\alpha_1(x_1,x_2) = x_1, \quad
	\alpha_2(x_1,x_2) = x_2, \quad
	\text{ and } \quad
	\alpha_3(x_1,x_2) = -x_1+x_2.
\]
By composing $\alpha_1$, $\alpha_2$, and $\alpha_3$ with a multiplexer, we can
define the network
\begin{equation}
	\label{eq:amn_max}
	\amn^\mathrm{max}(x_1,x_2)=
	\mu(\alpha_1(x_1,x_2),\alpha_2(x_1,x_2),\alpha_3(x_1,x_2)).
\end{equation}
It follows that $\amn^\mathrm{max}(x_1,x_2)$ evaluates to $\max(x_1,x_2)$,
for all $x\in\reals^2$. 
We often suppress the affine transformations $\alpha_i$ for notational
convenience, and record the expression~\eqref{eq:amn_max} directly as
$\amn^\mathrm{max}(x_1,x_2)=\mu(x_1,x_2,-x_1+x_2)$, see Figure~\ref{fig:max1}.
\qed
\end{example}

\begin{example}[Rectification]\label{ex:relu}
A common activation nonlinearity in neural networks is the rectifier (also
known as the ramp, or a rectified linear unit (ReLU)) function
$r:\reals\to\reals$ ($q=p=1$), where
$r(x)=\max(x,0)$.
Using Example~\ref{ex:max} with $x_2=0$ (constant) gives
the \acs{amn} 
$\amn^r(x)=\mu(x,0,-x)$, 
see Figure~\ref{fig:relu1}.
\qed
\end{example}

\begin{example}[Saturation]\label{ex:sat}
The saturation function $\sat:\reals\to\reals$ ($q=p=1$), defined as
\[
	\sat(x) = \left\{
	\begin{array}{ll}
		-1, & \text{if } x \leq -1,\\
		x, & \text{if } -1 < x < 1,\\
		1, & \text{otherwise},
	\end{array}
	\right.
\]
can be written as
$\amn^\mathrm{sat}(x)=\mu(1,\mu(-1,x,x+1),-x+1)$, see Figure~\ref{fig:sat1}.
\qed
\end{example}

\begin{figure}[htpb]
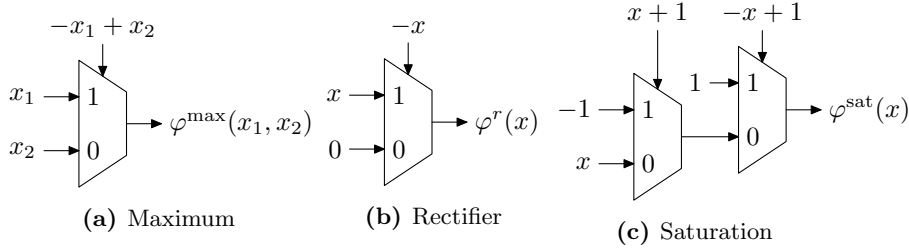

	\centering
	\hspace{-3em}
    \begin{subfigure}{0.33\linewidth}
		\centering
        \includegraphics[scale=1]{fig/mp/max-1}
		\caption{Maximum}
		\label{fig:max1}
    \end{subfigure}
    \begin{subfigure}{0.25\linewidth}
		\centering
        \includegraphics[scale=1]{fig/mp/relu-1}
		\caption{Rectifier}
        \label{fig:relu1}
    \end{subfigure}
    \begin{subfigure}{0.3\linewidth}
		\centering
        \includegraphics[scale=1]{fig/mp/sat-1}
		\caption{Saturation}
        \label{fig:sat1}
    \end{subfigure}
	\caption{Selected piecewise affine functions.}
    \label{fig:impl_pwl}
\end{figure}

\begin{example}[Smooth activations]
Functions that are not piecewise affine, like $f(x)=x^2$ (square), $f(x) =
(1+e^{-x})^{-1}$ (sigmoid), and $f(x)=\tanh(x)$, cannot be represented exactly
as an~\acs{amn}. They can, however, be approximated arbitrarily well by
an~\acs{amn} on any compact interval.
\qed
\end{example}

We now introduce a specific network as a running
example to demonstrate encoding and training of \acsp{amn}.

\begin{example}[Triplexer]\label{ex:triplex}
The \acs{amn} illustrated in Figure~\ref{fig:triplex1} is meant to resemble a
classical single-input/single-output, two-layer, feedforward networks.  It uses four
multiplexers, arranged in a feedforward topology, with two affine layers and
two nonlinear layers.

This triplexer~\acs{amn} 
$\amn^\mathsc{tri}_\theta : \reals\to\reals$ ($q=p=1$)
can be used to approximate a real-valued function.  It
is parameterized by the~24 weights and biases $\theta = (a_1,b_1,\ldots,f_4)
\in \reals^{24}$ making up 12 affine transformations:

\begin{align*}
	\text{First layer weights:}
	\quad
	&\left\{
		\begin{array}{l}
			x_1 := a_1x + b_1\\
			y_1 := c_1x + d_1\\
			z_1 := e_1x + f_1
		\end{array}
	\right.
	\quad\cdots\quad
	\left\{
		\begin{array}{l}
			x_3 := a_3x + b_3\\
			y_3 := c_3x + d_3\\
			z_3 := e_3x + f_3
		\end{array}
	\right.\\
	\text{First nonlinearity:}
	\quad
	&\left\{
		\begin{array}{l}
			w_1 := \mu(x_1, y_1, z_1)\\
			w_2 := \mu(x_2, y_2, z_2)\\
			w_3 := \mu(x_3, y_3, z_3)
		\end{array}
	\right.\\
	\text{Second layer weights:}
	\quad
	&\left\{
		\begin{array}{l}
			x_4 := a_4w_2 + b_4\\
			y_4 := c_4w_3 + d_4\\
			z_4 := e_4w_1 + f_4
		\end{array}
	\right.\\
	\text{Second nonlinearity:}
	\quad
	&\left\{
		\begin{array}{l}
			\vspace{-1ex}\\
			y := \mu(x_4,y_4,z_4)\\
			\vspace{-1ex}
		\end{array}
	\right.
	\qed
\end{align*}
\end{example}

\begin{figure}[htpb]
	\centering
	\includegraphics{fig/mp/triplex-1}
	\caption{The ``triplexer," a 2-layer 4-mux \acl{amn}.}
	\label{fig:triplex1}
\end{figure}

\paragraph{Discontinuous functions}
Following the previous examples, it is possible to express any
continuous, piecewise affine function exactly as an \acs{amn}. 
However, because the multiplexing function $\mu$ is effectively an if-then-else
statement, we can express many discontinuous functions as \acsp{amn} as well,
making these networks strictly more powerful for modeling
switched and hybrid systems than neural networks with continuous nonlinearities. 
The next example illustrates the powerful
modeling capability of \acsp{amn} in dynamical systems.


\begin{example}[Switched system]\label{ex:switched}
The dynamical system
with state dependent switching
\begin{equation}
	\label{eq:switched_lin}
	x(t+1) = \left\{
	\begin{array}{ll}
		A^- x(t), & \text{if } x_1(t) \leq 0,\\
		A^+ x(t), & \text{otherwise},
	\end{array}
	\right.
	\quad t = 0,1,2\ldots,
\end{equation}
where $A^-,A^+\in\reals^{n\times n}$ are given dynamics matrices, and
$x(t)\in\reals^n$ is the state at time $t$,
is equivalent to
the dynamical system
$x(t+1) = \amn^\mathrm{sw}(x(t))$. The state transition function
$\amn^\mathrm{sw}:\reals^n\to\reals^n$ 
is defined by the \acl{amn}
\[
	\amn^\mathrm{sw}(x) = \mu(A^- x, A^+ x, e_1^Tx),
\]
where $e_1 = (1,0,\ldots,0)$ is the standard basis vector in $\reals^n$.
\qed
\end{example}

Note that the (nonlinear) state transition function $\amn^\mathrm{sw}$ in
Example~\ref{ex:switched} need not be continuous on the switching boundary
$\{0\}\times \reals^{n-1}$; this discontinuity is in general impossible to
model exactly using a neural network with continuous (\eg, sigmoid, ReLU)
nonlinearities, but poses no difficulty for the \acs{amn} model, because the
multiplexing function $\mu$ (eq.~\eqref{eq:mu}) by design can be discontinuous
in its first two arguments.

\paragraph{Useful~\acsp{amn}}
Some common functions and their implementations appear in
Table~\ref{tab:impl_pwl}. Note that a valid~\acs{amn} need not necessarily be
convex, differentiable, or even continuous, as in the case of the cardinality
function~$\card(x)$. However an~\acs{amn} must ultimately be expressible as a
composition of multiplexers and affine transformations.
\begin{table}[hbtp]
	\centering
	\begin{tabular}{l|l|l}
		\textbf{Name} & \textbf{Function} & \textbf{AMN Expression}\\
		\hline
		maximum & $\max(x,y)$ & $\mu(x,y,-x+y)$\\
		minimum & $\min(x,y)$ & $\mu(y,x,-x+y)$\\
		rectification & $r(x)=\max(x,0)$ & $\mu(x,0,-x)$\\
		abs. value & $|x|$ & $\mu(-x,x,x)$\\
		saturation & $\sat(x)$ & $\mu(1,\mu(-1,x,x+1),-x+1)$\\
		deadzone & $\dz(x)$ & $\mu(x+1,\mu(x-1,0,x+1),-x+1)$\\
		$\|x\|_\infty$ & $\max(|x_1|,\max(|x_2|,\ldots))$ & $\mu(|x_1|,\mu(|x_2|\ldots),-|x_1|+\mu(|x_2|\ldots))$\\
		$\|x\|_1$ & $|x_1| + \cdots + |x_n|$ & $\sum_{i=1}^{n} \mu(-x_i, x_i, x_i)$\\
		$\card(x)$ & $|\{1\leq i\leq n \mid x_i \neq 0\}|$ & $\sum_{i=1}^{n}\mu(\mu(1,0,x_i),0,-x_i)$
	\end{tabular}
	\caption{Implementation of common functions as \aclp{amn}.}
	\label{tab:impl_pwl}
\end{table}

Note that by definition of an \acs{amn}, the enable condition for any
multiplexer is $z\leq 0$.  By composing multiplexers appropriately, it is
possible for the enable condition to have another real comparison,
\eg, $\leq, <, \geq, >, =, \neq$. For example, the enable condition $z \geq 0$
can be emulated by negating the enable input of a multiplexing unit, $\mu(x, y,
-z)$.  See Table~\ref{tab:impl_gates} for an \acs{amn} implementation of
allowable comparisons comparisons.

Multiple comparisons can be composed by gating operations (AND,
OR, XOR, \etc). For example, the \acs{amn} $\amn^\wedge(x,y,z_1,z_2)$
corresponds to the AND operation,
\[
	\amn^\wedge(x,y,z_1,z_2)
	=
	\left\{
		\begin{array}{ll}
			x, & \text{if}~z_1\leq 0 \text{ and}~z_2\leq 0,\\
			y, & \text{otherwise}.
		\end{array}
	\right.
\]
Refer to Table~\ref{tab:impl_gates} and Figure~\ref{fig:impl_gates} for a
visualization of a selected subset of these.


\begin{table}[htbp]
	\centering
	\begin{tabular}{r|l|l}
		\textbf{Gate} & \textbf{Definition} & \textbf{Expression}\\
		\hline
		AND & $\amn^\wedge(x,y,z_1,z_2)$ & $\mu(\mu(x,y,z_1),y,z_2)$\\
		OR & $\amn^\vee(x,y,z_1,z_2)$ & $\mu(x,\mu(x,y,z_1),z_2)$\\
		NOT & $\amn^\neg(x,y,z)$ & $\mu(y, x, z)$\\
		XOR & $\amn^{\oplus}(x,y,z_1,z_2)$ & $\mu(\mu(y,x,z_1),\mu(x,y,z_1),z_2)$ \\
		\hline
		LE & $\amn^\leq(x,y,z)$ & $\mu(x,y,z)$\\
		GE & $\amn^\geq(x,y,z)$ & $\mu(x,y,-z)$\\
		LT & $\amn^<(x,y,z)$ & $\amn^\neg(x,y,-z)$\\
		GT & $\amn^>(x,y,z)$ & $\amn^\neg(x,y,z)$\\
		EQ & $\amn^=(x,y,z)$ & $\amn^\wedge(x,y,z,-z)$\\
		NEQ & $\amn^{\neq}(x,y,z)$ & $\amn^\wedge(y,x,z,-z)$
	\end{tabular}
	\caption{Implementation of various gates/comparisons as AMNs.}
	\label{tab:impl_gates}
\end{table}

\begin{figure}[htpb]
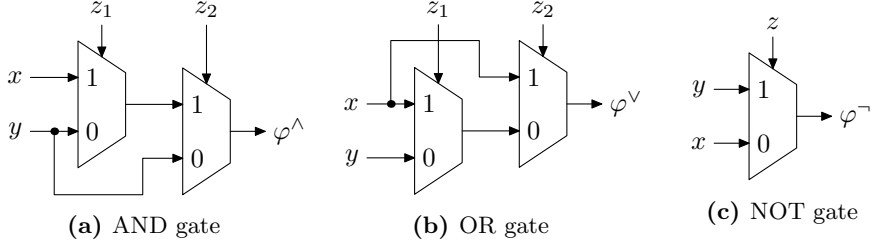

	\centering
    \begin{subfigure}{0.3\linewidth}
		\centering
        \includegraphics[scale=1]{fig/mp/and-1}
		\caption{AND gate}
		\label{fig:and1}
    \end{subfigure}
	~\quad~
    \begin{subfigure}{0.3\linewidth}
		\centering
        \includegraphics[scale=1]{fig/mp/or-1}
		\caption{OR gate}
        \label{fig:or1}
    \end{subfigure}
	~
    \begin{subfigure}{0.3\linewidth}
		\centering
        \includegraphics[scale=1]{fig/mp/not-1}
		\caption{NOT gate}
        \label{fig:not1}
    \end{subfigure}
	\caption{Selected gate functions.}
    \label{fig:impl_gates}
\end{figure}

\subsection{Key properties}
\paragraph{Well-definedness}
Not all compatibly dimensioned interconnections of multiplexers and affine
transformations result in an~\acs{amn} that is a well-defined function. A key
requirement on \acsp{amn} is a lack of variable dependence cycles.
Intuitively, phenomena such as race conditions cannot be resolved in a pure
mathematical structure like an~\acs{amn}---if an output of a component feeds
back to the input, then the output may not be uniquely determined by a given
input. We use the concept of a computation graph to make this concept formal.

\begin{definition}[Variables, direct dependency, computation graph, well-definedness]
Given an \acs{amn} $\amn$, the set $\Var(\amn)$ of internal signal
\emph{variables} is defined recursively as
\[
	\Var(\amn) = \left\{\hspace{-0.5em}
	\begin{array}{ll}
		\{x\}, & \text{if }\amn(x) = x,\\
		\{\amn\} \cup \Var(\amn_1), & \text{if }\amn(x) = \alpha(\amn_1(x)),\\
		\begin{aligned}
			\{\amn\} &\cup \Var(\amn_1)
			         \cup \Var(\amn_2) \\
					 &\cup \Var(\amn_3), 
		\end{aligned}	
		& \text{if }\amn(x)=\mu(\amn_1(x),\amn_2(x),\amn_3(x)).
	\end{array}
	\right.
\]
For two variables $v_i,v_j\in\Var(\amn)$, we say that $v_j$
\emph{directly depends} on $v_i$ if 
$v_j=\alpha(v_i)$
for some $\alpha$, or 
$v_j=\mu(v_k,v_l,v_m)$, with $i\in\{k,l,m\}$.
The \emph{computation graph} of $\amn$ is a directed graph $G(\amn)=(V,E)$,
with vertices $V=\Var(\amn)$,
and an edge for every direct dependency, 
$E=\{(v_i,v_j) \in V\times V \mid v_j \text{ directly depends on } v_i\}$. 
The network $\amn$ is \emph{well-defined} if $G(\amn)$ has no directed cycles.
\end{definition}

\begin{definition}[Inputs and outputs]
Given a well-defined \acs{amn} $\amn$, and its computation graph
$G(\varphi)=(V,E)$, the sets of \emph{input} and \emph{output} variables are,
respectively
\begin{align*}
	\Input(\varphi) &= \{ v \in V \mid 
		v \text{ has no incoming edges}\},\\
	\Output(\varphi) &= \{ w \in V \mid 
		w \text{ has no outgoing edges}\}.
\end{align*}
\end{definition}

The computation graph encodes dependency relationships between the internal
variables.  The set of variables $\Var(\amn)$ contains unique names for the
outputs of the constituent units of $\amn$. 
Following the electronic systems metaphor, the variables are the
\emph{signals} or \emph{wires} in a circuit diagram like
Figures~\ref{fig:basicblocks}--\ref{fig:triplex1}; variables also correspond to
the \emph{vertices} or \emph{nodes} of $G(\amn)$. Similarly, the computational
operations $\mu$ and $\alpha$ correspond to \emph{edges} in $G(\amn)$.
The nodes of $G(\amn)$ with no incoming edges are inputs to $\amn$ (or
constants), and nodes with no outgoing edges are the outputs of $\amn$. 
For a given assignment to the inputs, there is a unique assignment to all
internal nodes and the outputs, provided the network $\amn$ is well-defined.

The computation graph $G(\amn^\mathsc{tri}_\theta)$ for the triplexer from
Example~\ref{ex:triplex} is shown in Figure~\ref{fig:trigraph1}.
It depicts a natural flow of information for computing a real output
$y=\amn^\mathsc{tri}_\theta(x)$ for a given real input $x$.
In this case, the number of variables is
$|V|=17$ with $V = \{x,z_1,x_1,y_1,\ldots,y\}$,
and
$|E|=24$ with $E = \{ (x,z_1),(x,x_1),(x,y_1),\ldots,(y_4,y) \}$. 
Since there are no directed cycles in $G(\amn^\mathsc{tri}_\theta)$, the
expression $\amn^\mathsc{tri}_\theta$ is a well-defined function from $\reals$
to $\reals$.

\begin{figure}[htpb]
	\centering
	\includegraphics{fig/mp/trigraph-1}
	\caption{Computation graph $G(\amn^\mathsc{tri}_\theta)$ for the triplexer.}
	\label{fig:trigraph1}
\end{figure}

\paragraph{Non-uniqueness}
There is usually more than one way to express a given
piecewise affine function as an \acs{amn}. For example, using the identity
$\sat(x)=r(x+1)-r(x-1)-1$, we obtain an alternative expression
\begin{align*}
	\amn^{\mathrm{sat}'}(x) 
	&= \amn^r(x+1)-\amn^r(x-1)-1\\
	&= \mu(x+1,0,-x-1)-\mu(x-1,0,-x+1)-1,
\end{align*}
which is distinct from, and has a different encoding than the expression
for~$\amn^\mathrm{sat}$ in Example~\ref{ex:sat}. Nevertheless the two
expressions evaluate to the same output, \ie, $\amn^\mathrm{sat}(x) =
\amn^{\mathrm{sat}'}(x)$, for all $x\in\reals$. 
This non-uniqueness means that one implementation of a given function as an
\acs{amn} can be more efficient than another if that implementation uses fewer
multiplexers or affine transformations, or if those multiplexers and affine
transformations have a smaller dimensionality.

\paragraph{Universal approximation}
Like a classical neural network, an AMN can approximate an arbitrary nonlinear
function, but unlike classical neural networks, the AMN output is allowed to be
discontinuous. AMNs inherit the universal function approximation property from
the neural networks they embed (\cf~\cite{Cybenko:1989}).

\subsection{Encoding}\label{sec:encoding}
By focusing on affine transformations and affine enable conditions, we
constrain the network input-output relationship to be a conjunction or
disjunction of statements over the linear real arithmetic. If the
inputs and outputs are unbound or partially bound, a satisfying
assignment can be obtained by solving a sequence of linear programs (LPs). To
see how this works, we first define a recursive procedure for converting an
\acl{amn} $\amn$ into statements over linear real arithmetic (SMT encoding), or
a set of linear constraints over real and binary variables (MIP encoding).

\paragraph{SMT encoding}
For a given AMN $\amn$ with input $x\in\reals^q$ and output
$y\in\reals^p$, the formula $\smt_\amn[x,y]$ is a first-order logic formula
given recursively as
\begin{align}
	\smt_{x}[x,y]
	&\equiv
	\left\{
		y = x
	\right\},\label{eq:smt_enc_var}
	\\
	\smt_{\alpha(\amn_1)}[x,y]
	&\equiv 
	\left\{
		\begin{aligned}
			\exists v.\, 
			&(y = A v + b) \\
			& \wedge \smt_{\amn_1}[x,v]
		\end{aligned}
	\right\},
	\text{ where }
	\alpha(\xi) = A \xi + b,\label{eq:smt_enc_aff}
	\\
	\smt_{\mu(\amn_1,\amn_2,\amn_3)}[x,y]
	&\equiv 
	\left\{
		\begin{aligned}
		\exists u, v, w.\,
			& ((w \leq 0) \rightarrow (y = u)) \\
			& \wedge (\neg(w \leq 0) \rightarrow (y = v)) \\
			& \wedge \smt_{\amn_1}[x, u] \\
			& \wedge \smt_{\amn_2}[x, v] \\
			& \wedge \smt_{\amn_3}[x, w]
		\end{aligned}
	\right\}.\label{eq:smt_enc_mu}
\end{align}

\begin{example}[Triplexer (SMT)]
After simplification and variable renaming, we obtain a linear real
arithmetic encoding of the triplexer from from Example~\ref{ex:triplex}:
\[
	\smt_{\amn_\theta^\mathsc{tri}}[x,y]
	\equiv
	\left\{
	\begin{aligned}
	\exists &(x_1,y_1,z_1,\ldots,x_4,y_4,z_4,w_1,w_2,w_3)\in\reals^{15}.\\
	& \bigwedge_{i=1}^{3} 
		(x_i = a_i x + b_i \wedge
		y_i = c_i x + d_i \wedge
		z_i = e_i x + f_i)\\
	& \wedge \bigwedge_{j=1}^{3} 
		((z_j \leq 0) \rightarrow (w_j=x_j))
		\wedge (\neg (z_j \leq 0) \rightarrow (w_j=y_j))\\
	& \wedge (x_4 = a_4 w_2 + b_4 \wedge
		y_4 = c_4 w_3 + d_4 \wedge
		z_4 = e_4 w_1 + f_4) \\
	& \wedge 
		((z_4 \leq 0) \rightarrow (y=x_4))
		\wedge (\neg (z_4 \leq 0) \rightarrow (y=y_4))
	\end{aligned}
	\right\}.
\]
Note that the only unbound variables in $\smt_{\amn_\theta^\mathsc{tri}}[x,y]$
are the input $x$ and output $y$. Furthermore, every clause is an affine
equation, inequality, or the logical negation of an affine equation or
inequality.
\qed
\end{example}

\paragraph{MIP encoding}
Given an \acs{amn} $\amn$, and its computation graph $G(\amn)=(V,E)$, we define
a collection of mixed integer constraints, parameterized by $x\in\reals^q$ and
$y\in\reals^p$, over the variables $\Var(\amn)$ and additional binary
variables as follows:
\begin{enumerate}
	\item For $v\in\Input(\amn)$, add the constraint $x=v$;
	for $w\in\Output(\amn)$, add $y=w$.
	\item For each $(v_i, v_j) \in E$ with $v_j = \alpha(v_i)$,
	$\alpha(\xi) = A \xi + b$, add the constraint $v_j = A v_i + b$ with real
	(vector) variables $v_i$ and $v_j$.
	\item For each $(v_k,v_j), (v_l, v_j), (v_m, v_j) \in E$ with $v_j = \mu(v_k,
	v_l, v_m)$, add the mixed integer ``big-$M$" constraints
	\begin{equation}
		\label{eq:mip_enc_mu}
		\begin{aligned}
			-Mb_j &< v_m \leq M(1-b_j),\\
			-\ones Mb_j &\preceq v_j - v_l \preceq \ones Mb_j,\\
			-\ones M(1-b_j) &\preceq v_j - v_k \preceq \ones M(1-b_j), \quad b_j\in\{0,1\},
		\end{aligned}
	\end{equation}
	with real (vector) variables $v_j,v_k,v_l,v_m$ and binary variables $b_j\in\{0,1\}$.
\end{enumerate}

The formula $\mip_\amn[x,y]$ is the conjunction of the constraints
obtained through these steps.
Every affine unit in $\amn$ corresponds to an affine equality
constraint, and every multiplexer corresponds to a binary variable $b_j$ that
is true ($b_j=1$) if and only if the corresponding enable condition is met
($v_m\leq 0$). 
We use the ``big-$M$" constraints~\eqref{eq:mip_enc_mu}, which are equivalent
to the constraint $v_j = \mu(v_k, v_l, v_m)$, provided $M$ is a large enough
constant, see~\cite{Grossmann:2002}.

\subsection{Optimization}
Recall that equality constraints like $y=h(x)$, with variables $x$ and $y$, can
be efficiently imposed in linear programs, and in general, in convex
optimization programs (see, \eg~\cite{Boyd:2004}).
The aim of encoding an \acs{amn} in SMT or MIP is to
represent constraints like $y = \amn(x)$, with $x$ and $y$ as variables, and
$\amn$ is an arbitrary, possibly non-affine \acs{amn}, in an optimization problem. 
Ultimately, we would like to be able to solve optimization programs in the form
\begin{equation}
	\label{eq:optimization}
	\begin{array}{ll}
		\mbox{minimize}   & \amn_0(x) \\
		\mbox{subject to} 
			& \amn_i(x) \leq 0, \quad i=1,\ldots,m_1, \\
			& \psi_j(x) = 0, \quad j=1,\ldots,m_2,
	\end{array}
\end{equation}
over a variable $x\in\reals^q$,
where $\amn_0,\ldots,\amn_{m_1},\psi_1,\ldots,\psi_{m_2}$ are arbitrary
\acsp{amn}. The idea is made more clear by the graph of an~\acs{amn}.

\begin{definition}[Graph]
The \emph{graph} of a function $f:\reals^q\to\reals^p$ is the set of
input-output pairs $\Graph{f}=\{(x,y)\in\reals^q\times\reals^p\mid y=f(x)\}$.
For a first-order logic formula $\psi[x,y]$ with free variables $x$, $y$, it is
the set of satisfying assignments, $\Graph{\psi} =
\{(c_1,c_2)\in\reals^q\times\reals^p\mid\psi[x:=c_1,y:=c_2]\}$.
\end{definition}

\begin{example}
For the real-valued function $f(x) = x^2$, and the first-order logic
formula 
$\psi[x,y]\equiv\exists z .\; (x\geq z) \wedge (y \geq 0)$,
we have
$\Graph{f} = \Graph{\psi} = \reals\times\reals_+$.
\end{example}

\begin{theorem}[Encoding]\label{thm:encoding}
Given a well-defined \acs{amn} $\amn$,
\[
\Graph{\amn} = \Graph{\smt_\amn[x,y]} = \Graph{\mip_\amn[x,y]}.
\]
\end{theorem}
\begin{proof}
By construction.
\end{proof}

\paragraph{Querying the solver}
As a consequence of Theorem~\ref{thm:encoding}, we can represent the graph
$\Graph{\amn}$ of an \acs{amn} $\amn$ as a conjunction or disjunction of linear
atoms. We can formulate nonconvex feasibility and optimization problems over
real vector variables, and solve them using an SMT or MIP solver. 

For example, the (nonconvex) feasibility problem 
\begin{equation}
	\label{eq:feasibility}
	\begin{array}{ll}
		\mbox{find}   & (x,y)\in\reals^{q\times p} \\
		\mbox{subject to} & y = \amn(x)
	\end{array}
\end{equation}
has a solution if and only if there exists $(x,y)\in\Graph{\amn}$. 
In other words, the problem~\eqref{eq:feasibility} is equivalent to the problem
\begin{equation*}
	\label{eq:feasibility2}
	\begin{array}{ll}
		\mbox{find}   & (x,y)\in\reals^{q\times p} \\
		\mbox{subject to} & (x,y) \in \Graph{\amn},
	\end{array}
\end{equation*}
which can be found by posing the query
\begin{equation}
	\label{eq:smt_feasibility}
	\exists x \in \reals^q \, . \,
	\exists y \in \reals^p \, . \,
	\smt_\amn[x,y]
\end{equation}
to an SMT solver, or replacing the constraint $y=\amn(x)$ with $\mip_\amn[x,y]$
in a MIP solver. Moreover, the procedure for translating the feasibility
problem~\eqref{eq:feasibility} into the formal
problem~\eqref{eq:smt_feasibility} suitable for an SMT solver using the linear
theory is entirely constructive and mechanical, following the steps outlined in
\S{\ref{sec:encoding}}, and complete in the sense that the
problem~\eqref{eq:feasibility} is feasible if and only
if~\eqref{eq:smt_feasibility} is \textsc{sat}. Conversely, the
problem~\eqref{eq:feasibility} is infeasible if and only if the
query~\eqref{eq:smt_feasibility} is \textsc{unsat}.

\paragraph{Constrained optimization using bisection}
Constrained optimization involving \acsp{amn} can be accomplished by a sequence
of feasibility queries to an SMT or MIP solver. For example, consider the
simplified (nonconvex) optimization problem
\begin{equation}
	\label{eq:optimization2}
	\begin{array}{ll}
		\mbox{minimize}   & \amn_0(x) \\
		\mbox{subject to} & \amn_i(x) \leq 0, \quad i=1,\ldots,m
	\end{array}
\end{equation}
over a variable $x\in\reals^q$, where $\amn_i : \reals^q\to\reals$,
$i=0,\ldots,m$ are \acsp{amn}. The optimal objective value satisfies
$\amn_0^\star \leq t$ if and only if the SMT query
\begin{equation}
	\label{eq:smt_optimization}
	\exists x \in \reals^q \, . \,
	\exists y_0,\ldots,y_m \in \reals \, . \,
	\bigwedge_{i=0}^{m} \smt_{\amn_i}[x,y_i] 
	\wedge
	\bigwedge_{i=1}^{m} (y_i \leq 0)
	\wedge
	(y_0 \leq t)
\end{equation}
is \textsc{sat}. We can minimize the function $\amn_0(x)$ by bisection on $t$
through a sequence of feasibility calls of the
form~\eqref{eq:smt_optimization}, see~\cite[\S{4.2.5}]{Boyd:2004}. If the
initial interval $[l,u]$ contains $\amn^\star_0$, then the number of SMT
calls needed to compute $\amn_0^\star$ to tolerance $\epsilon > 0$ using the
bisection method is at most $\lceil\log_2((u-l)/\epsilon) \rceil$. Moreover, an
$\epsilon$-suboptimal value $x^\star_\epsilon$ with $|\amn_0(x^\star_\epsilon)
- \amn_0^\star| \leq \epsilon$ is obtained directly from the last
query~\eqref{eq:smt_optimization} that returned \textsc{sat}. 

\paragraph{AMNET modeling toolbox}
The bisection procedure is implemented in our open-source modeling package,
\amnet\footnote{\texttt{https://github.com/ipapusha/amnet}}. In addition to
allowing one to define and evaluate \acsp{amn} with the building blocks $\mu$
and $\alpha$, \amnet\ allows one to define new \acsp{amn} by composing existing
\acsp{amn} in a disciplined manner, automatically convert neural networks from
to \acsp{amn}, and solve optimization problems with \acs{amn} objectives and
constraints. See \S\ref{sec:examples} for example applications.

\subsection{Training}
\subsubsection{Gradient descent}
Finding weights of an AMN to solve a regression or classification problem can
be accomplished with a modified version of the gradient descent algorithm.  In
such problems the goal is to find weights $\theta$ to minimize an objective
$J(\theta)$; example objectives could be least squares or negative
log-likelihood.

For illustration, consider a ``perceptron" network
$\amn^\mathsc{per}_\theta:\reals^n\to\reals$ consisting of a single affine
layer with a multiplexing nonlinearity,
\[
	\amn^\mathsc{per}_\theta(x) = \mu(\alpha(x), \beta(x), \gamma(x)),
	\quad
	\left\{
	\begin{array}{l}
		\alpha(x) = a^T x + b,\\
		\beta(x) = c^T x + d,\\
		\gamma(x) = e^T x + f,
	\end{array}
	\right.
\]
where
$\theta=(a,b,c,d,e,f)\in\reals^n\times\reals\times\cdots\times\reals^n\times\reals$
are the weights parameterizing the network.

To apply gradient descent, it is necessary to compute a descent direction
$\nabla_\theta J(\theta)$, which requires computing the gradient of the
\acs{amn} with respect to its weights. However, the function 
$\amn^\mathsc{per}_\theta(x)$ is not necessarily differentiable in $\theta$. It
is differentiable almost everywhere, except on the set $\{\theta \mid e^Tx + f
= 0 \}$ having measure zero, because the multiplexing nonlinearity $\mu$ is not
necessarily continuous there.
However the (weak) derivative of the nonlinearity $\mu$ can be written in terms
of the nonlinearity $\mu$ itself,
\begin{align}
	\nabla_\theta (\amn^\mathsc{per}_\theta(x))
	&= (
		\nabla_a (\amn^\mathsc{per}_\theta(x)),
		\nabla_b (\amn^\mathsc{per}_\theta(x)),
		\ldots,
		\nabla_f (\amn^\mathsc{per}_\theta(x))
		)\nonumber \\
	&= \begin{bmatrix}
		\mu(x,0,\gamma(x))\\
		\mu(1,0,\gamma(x))\\
		\mu(0,x,\gamma(x))\\
		\mu(0,1,\gamma(x))\\
		0\\
		0
	\end{bmatrix}, \label{eq:weak_deriv}
\end{align}
allowing us to define a version of backpropagation where the nondifferentiable 
weights do not change.
\begin{tabbing}
    \rule{\linewidth}{0.4pt}\\
    \textbf{algorithm:} Gradient descent for AMNs\\
	\textbf{given}: an AMN, training objective $J(\theta)$, $k=0$, initial
	$\theta^{(0)}$, learning rates $\alpha_k$, \\
	\quad
		\= tolerance $\epsilon > 0$\\
    \textbf{repeat}: \\
        \> 1. \= $\theta^{(k+1)} := \theta^{(k)} - \alpha_k \nabla_\theta(J(\theta^{(k)}))$\\
		\> 2. \> $k:=k+1$\\
    \textbf{until}: $\|\nabla_\theta(J(\theta^{(k)}))\|\leq \epsilon$
	(or another stopping criterion)\\
    \rule{\linewidth}{0.4pt}
\end{tabbing}

Performance modifications to the gradient descent algorithm, including
momentum, batching, and dropout, are ready extensions.
One deficiency of gradient descent stems from the enable parameters ($e$ and
$f$ in $\amn^\mathsc{per}$ above) not changing with training. The network is
stuck with the initial weights and biases parameterizing the enable components
of $\theta$, because those derivatives are set to zero,
see eq.~\eqref{eq:weak_deriv}.  However,
an AMN trained by backpropagation in our experiments can still be remarkably
expressive if the stuck parameters are initially well-dispersed. 

In the next section, we will use the SMT encoding to suggest a novel---though
much less efficient---algorithm that trains all AMN parameters at the same time
without stuck enable weights.


\subsubsection{SMT embedding in NL theory}
\begin{definition}(Dual of an AMN)
Given an \acs{amn} $\amn_\theta : \reals^q \to \reals^p$, where
$\theta\in\reals^r$ are the parameters, \ie, (stacked) weights and biases of
all the affine units, its dual is a function $\amn^\circ : \reals^r
\to\reals^p$ such that $\amn_x^\circ(\theta) = \amn_\theta(x)$.
\end{definition}

In other words, the dual $\amn^\circ$ is the same as the original
\acs{amn} $\amn$ with the roles of the input variable $x$ and parameters
$\theta$ reversed. An encoding of $\amn^\circ$ can be obtained from the SMT
encoding $\smt_\amn[x,y]$ by adding an existential quantifier for every
component of $\theta$ and assigning $x$.
We are being intentionally agnostic about the stacking and ordering of the
weights and biases of a network $\amn$ into a parameter vector $\theta$,
because there can be many consistent ways to do it. 

\begin{example}[Dual of perceptron]
The dual of the single-layer ``perceptron" network
\[
	\amn_\theta(x) = \mu(\alpha(x), \beta(x), \gamma(x)),
	\quad
	\left\{
	\begin{array}{l}
		\alpha(x) = a^T x + b,\\
		\beta(x) = c^T x + d,\\
		\gamma(x) = e^T x + f,
	\end{array}
	\right.
\]
where
$\theta=(a,b,c,d,e,f)\in\reals^n\times\reals\times\cdots\times\reals^n\times\reals$
are the weights parameterizing the network, is given by 
\[
	\amn^\circ_x(\theta) = \mu(
		\alpha^\circ_x(\theta), 
		\beta^\circ_x(\theta), 
		\gamma^\circ_x(\theta)),
	\quad
	\left\{
	\begin{array}{l}
		\alpha^\circ_x(\theta) = x^T a + b,\\
		\beta^\circ_x(\theta) = x^T c + d,\\
		\gamma^\circ_x(\theta) = x^T e + f,
	\end{array}
	\right.
\]
where $x$ is the (fixed) parameter and $\theta=(a,b,c,d,e,f)$ is the variable.
\qed	
\end{example}

The dual $\amn^\circ$ in the previous example is itself an AMN, but this need
not be the case in general.  As a result, the relation between the input and
output of $\amn^\circ$ cannot always be encoded in SMT using a linear theory.
For example, by interchanging the roles of the parameters and the input, the
dual of the triplexer (Example~\ref{ex:triplex}) involves products between
existentially quantified variables, \eg, $a_4$ and $w_2$. 

Thus, by following the SMT encoding procedure with the nonlinear theory (NL),
we can still find elements of $\Graph{\amn^\circ}$ in a decidable way. However,
working with the nonlinear theory is much less computationally efficient than
working with the linear theory.

We use the concept of a dual AMN to propose a training procedure for the
weights and biases by encoding them as variables in an SMT query, and using
consistency training.

\paragraph{Consistency training}
Let $\mathcal{D}= \{(x^{(i)}, y^{(i)}\}_{i=1}^{N}$ be a data set of ordered
pairs of training data points. 
Given an \acs{amn} $\amn_\theta$, we say that the network is
$\epsilon$-\emph{consistent} with the data point $(x,y)\in\mathcal{D}$ if
$\|y - \amn_\theta(x)\|_1 \leq \epsilon$. 
The set of parameters defining all $\epsilon$-consistent networks,
\[
	\Theta_\epsilon = \{\theta \in \reals^r \mid \|y - \amn_\theta(x) \|_1 \leq
	\epsilon \text{ for all } (x,y)\in\mathcal{D}\},
\]
is SMT representable in the nonlinear theory. The encoding of $\Theta_\epsilon$
can be derived by observing that for a fixed $\epsilon \geq 0$, the set
$\Theta_\epsilon$ is nonempty if and only if
\begin{equation}
	\label{eq:cons_training}
	\exists \theta\in\reals^r\, . \bigwedge_{(x,y)\in\mathcal{D}} \|y -
	\amn_x^\circ(\theta)\|_1 \leq \epsilon
\end{equation}
is satisfiable. The expression~\eqref{eq:cons_training} is SMT representable,
because each conjunct is SMT representable.

There are many other ways to define consistency. A weighted sum of norms can be
used to emphasize certain training samples. Each conjunct involves a $1$-norm
here, but it is possible to use, \eg, a weighted norm, a ramp function
(regret), or any other SMT representable function. By restricting to general
$p$-norms, all constraints remain polynomial equations and inequalities.

The consistency viewpoint is useful because we can \emph{train} or
\emph{retrain} a network $\amn_\theta$ by posing the
query~\eqref{eq:cons_training} to an SMT solver. If the result is \textsc{sat},
then the parameter $\theta_0\in\Theta_\epsilon$ obtained from the query defines
a network $\amn_{\theta_0}$ with which every point in $\mathcal{D}$ is
$\epsilon$-consistent. If the result is \textsc{unsat}, then no
assignment to the network parameters can result in an $\epsilon$-consistent
network for the whole data set. In such a case, we can either remove offending
examples from $\mathcal{D}$, increase $\epsilon$, or make the network more
expressive by adding extra architectural layers. A similar data set consistency
idea was used in~\cite{Papusha:2018a} to solve inverse optimal control problems
with regular language specifications.

\paragraph{Robust consistency training}
Suppose we would like to train a network that is robust with respect to, \eg,
bounded (rectangular) perturbations on the input. In the consistency framework,
this might correspond to the query
\begin{equation}
	\label{eq:rob_cons_training}
	\exists \theta\in\reals^r\, . 
	\forall \delta\in[\delta^-,\delta^+]^q\, .
	\bigwedge_{(x,y)\in\mathcal{D}} 
		\|y - \amn_{x+\delta}^\circ(\theta)\|_1 \leq \epsilon.
\end{equation}
The intuition is this: since $(x+\delta)$ and $\theta$ multiply together in the
encoding
of~\eqref{eq:rob_cons_training}, and the query~\eqref{eq:rob_cons_training} is
in \emph{exists-forall} (EF) form, we can synthesize a robust
$\epsilon$-consistent network by a sequence of queries to an SMT solver; see,
\eg~\cite{Cheng:2013}.


\clearpage
\section{AMNs in the loop}\label{sec:inloop}
\subsection{Autonomous stability}\label{sec:autonomous_stability}
In this section, we are concerned with defining a procedure that formally and
automatically proves properties of the autonomous discrete time nonlinear
system
\begin{equation}
	\label{eq:dtsys}
	x(t+1) = \amn(x(t)), 
	\quad x(0) = x_0,
	\quad t=0,1,2,\ldots,
\end{equation}
where $x(t)\in\reals^n$ is the state at time $t$,
$x_0\in\mathcal{X}_0\subseteq\reals^n$ is the initial condition, $\mathcal{X}_0$ is
the initial set, and $\amn:\reals^n\to\reals^n$ is any well-defined \acs{amn}. 

\begin{example}[Neural networks in the loop]
A known, memoryless state-feedback neural network controller $u(t) =
\nn(x(t))$, $\nn : \reals^{n} \to \reals^{m}$, which uses only piecewise affine
nonlinearities (\eg, ReLU), stabilizes the linear system
\begin{equation}
	x(t+1) = Ax(t) + Bu(t), 
	\quad u(t) = \nn(x(t)), \quad t = 0,1,\ldots,
\end{equation}
if and only if the system~\eqref{eq:dtsys} is stable
with $\amn(x) = Ax + B\cdot\nn(x)$.
The system is illustrated in Figure~\ref{fig:nnio1}.
\qed
\end{example}

\begin{figure}[htpb]
	\centering
	\includegraphics{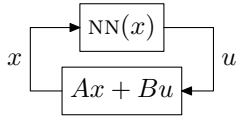}
	\caption{Neural network in the loop.}
	\label{fig:nnio1}
\end{figure}

\begin{example}[Variable-gain control]
The phase-based variable-gain nonlinearity from~\cite[eq.~(4)]{Hunnekens:2016},
\[
	\varphi(e,\dot{e}) = \left\{
	\begin{array}{ll}
		\alpha e, & \text{if } e\dot{e} > 0,\\
		0, & \text{otherwise},
	\end{array}
	\right.
\]
can be written as the network
$\amn^\mathrm{vgc}(e,\dot{e})
=\amn^\vee(\amn^\vee(0,\alpha e, -e, -\dot{e}),\alpha e, e, \dot{e})$, (see
Table~\ref{tab:impl_gates}).  The network satisfies
$\amn^\mathrm{vgc}(e,\dot{e}) = \varphi(e,\dot{e})$ for all $e$, $\dot{e}$.
\qed
\end{example}

\begin{example}[Linear feedback system]
With $u(t) = \nn(x(t)) = K x(t)$, where $K \in \reals^{m\times n}$, the
system~\eqref{eq:dtsys} corresponds to a linear feedback controller. Note that
we can write $\nn(x(t))$ as single-layer neural network with at most $2 m$
internal units and $2 m n$ weights by taking advantage of the identity $x=r(x)
- r(-x)$. The weights of the neural network are $K$ and $-K$.
\qed
\end{example}

\begin{example}[Saturation nonlinearity]
The single-input system~\cite[ex~2.1]{Johansson:2003}
\[
	x(t+1) = Ax(t) + b\cdot\sat(v(t)), 
	\quad v(t)=k^T x(t),
	\quad t = 0,1,\ldots
\]
with saturation, where the saturation function is
\[
	\sat(v) = \left\{
	\begin{array}{ll}
		-1, & v \leq -1\\
		v, & -1 < v < 1\\
		1, & v \geq 1,
	\end{array}
	\right.
\]
can be written as a linear system with neural network feedback.  The saturation
function is an affine combination of ReLU nonlinearities,
$\sat(x)=r(x+1)-r(x-1)-1$.  In this case, the neural network function is given
by $\nn:\reals^n \to \reals$, 
\[
	\nn(x(t)) 
	= \sat(k^T x(t)) 
	= r(k^Tx(t) + 1) - r(k^Tx(t) - 1) - 1.
\]
The controller weights $k\in\reals^m$ are encoded as weights in the neural
network.
\qed
\end{example}

\subsection{Counterexample-guided {L}yapunov search}
To prove autonomous stability for systems like ones in the previous section
\S\ref{sec:autonomous_stability}, we will synthesize a Lyapunov function, which
is positive definite, and decreases along trajectories of the
system~\eqref{eq:dtsys}.
The following procedure searches for a Lyapunov function $V:\reals^n\to\reals$
from a candidate class $\mathcal{V}$ by keeping track of a \emph{counterexample
set} $\mathcal{C}\subseteq\mathcal{X}_0$ of initial points.
Early work in this direction includes \cite{Cheng:2013,Kapinski:2014,Ravanbakhsh:2015}.

\begin{enumerate}
	\item \label{enum:candidate} \emph{Select candidate}: Choose a candidate
	Lyapunov function $V$ from $\mathcal{V}$ that decreases on $\mathcal{C}$,
	\ie, $V$ should satisfy
	\[
		\forall x_0 \in \mathcal{C} .\;
		(V(0) = 0) 
		\wedge (x_0 \neq 0 \rightarrow V(x_0) > 0)
		\wedge (V(\amn(x_0)) - V(x_0) < 0).
	\]
	If such a $V$ cannot be found within $\mathcal{V}$, return \textsc{unknown}.
	\item \label{enum:counterexample} \emph{Generate counterexample}: Attempt
	to find a point $x_c\in\mathcal{X}\setminus\mathcal{C}$ at which the
	candidate $V$ fails to decrease, \ie, pose the query
	\[
		\exists x_c \in\mathcal{X} .\;
		(x_c \neq 0 \wedge V(x_c) \leq 0) 
		\vee
		(V(\amn(x_c)) - V(x_c) \geq 0).
	\]
	\item \label{enum:update} \emph{Update}: If such $x_c$ is found, update
	$\mathcal{C}:=\mathcal{C}\cup\{x_c\}$ and go to step~1. Otherwise, return
	\textsc{stable}.
\end{enumerate}

If the first step fails, we must terminate the procedure, and say nothing about
the asymptotic stability or instability of~\eqref{eq:dtsys}, unless
$\mathcal{V}$ is known to be expressive enough, \eg, polyhedral for certain
$\amn$~\cite{Bitsoris:1988}.
Furthermore, the procedure need not terminate.  The art and science of this
counterexample-guided prescription is in choosing a computationally tractable
$\mathcal{V}$ and update procedures.

For example, let $\mathcal{A}_N$ be the set of all well-defined \acs{amn}s with
at most $N$ multiplexers. If $\mathcal{V}\subseteq\mathcal{A}_N$, and
$\mathcal{X}$ is SMT representable, then Step~\ref{enum:counterexample} is
equivalent to the query
\[
	\begin{aligned}
	& \exists x_c\in\mathcal{X},\,\exists x_c^+ \in\reals^n,\, \exists v_c, v_c^+\in\reals .\;
		\smt_{V}[x_c,v_c] \wedge \big((x_c \neq 0 \wedge v_c \leq 0)\\
	& \quad \vee 
		(\smt_{\amn}[x_c,x_c^+] \wedge \smt_{V}[x_c^+,v_c^+]
		 \wedge v_c^+ - v_c \geq 0)\big),
	\end{aligned}
\]
where we used Theorem~\ref{thm:encoding} to rewrite the constraint
\[
	x_c^+ = \amn(x_c) 
	\Longleftrightarrow 
	(x_c, x_c^+) \in \Graph{\amn}
	\Longleftrightarrow
	\smt_{\amn}[x_c,x_c^+],
\]
and similarly,
\[
	v_c = V(x_c)
	\Longleftrightarrow 
	(x_c, v_c) \in \Graph{V}
	\Longleftrightarrow
	\smt_{V}[x_c,v_c].
\]

\paragraph{General framework}
In general, we can often abstract questions about control systems to searches
for a Lyapunov function satisfying a stability property,
\begin{equation}
	\label{eq:ef_lyapunov}
	\exists V \in \mathcal{V} .\, \forall x\in\mathcal{X} .\;
	\mathsc{lyap}(V,x),
\end{equation}
where $\mathsc{lyap}(V,x)$ is a formula that includes relevant Lyapunov
stability conditions. For example the formula $\mathsc{lyap}$ might be:
\begin{itemize}
	\item \emph{Global stability}: $\mathcal{X}=\reals^n$
	\[
		\mathsc{lyap}(V,x)
		\equiv
		(V(0) = 0)
		\wedge (x \neq 0 \rightarrow V(x) > 0) 
		\wedge (V(x^+) - V(x) < 0)
	\]
	\item \emph{Region of attraction}: $\mathcal{X}\subseteq\reals^n$
	\[
		\mathsc{lyap}(V,x)
		\equiv
		(V(0) = 0)
		\wedge (x \neq 0 \rightarrow V(x) > 0) 
		\wedge (V(x^+) - V(x) < 0)
	\]
	\item \emph{Decay rate}:
	\[
		\mathsc{lyap}(V,x)
		\equiv
		(V(0) = 0)
		\wedge (x \neq 0 \rightarrow V(x) > 0) 
		\wedge (V(x^+) - \gamma V(x) \leq 0)
	\]
	\item \emph{Positively invariant set}:
	\[
		\mathsc{lyap}(V,x)
		\equiv
		(V(0) = 0)
		\wedge (x \neq 0 \rightarrow V(x) > 0) 
		\wedge (V(x) \leq 0 \rightarrow V(x^+) \leq 0)
	\]
\end{itemize}

The problem~\eqref{eq:ef_lyapunov} is in exists-forall form, and can be tackled
by the general counterexample-guided procedure below, provided that the
following subprocedures are tractable:
\begin{align*}
	\mathsc{E-solve}(\mathcal{X}_c)
	&\equiv \exists V \in \mathcal{V} .\, \bigwedge_{x_c \in \mathcal{X}_c}
		\mathsc{lyap}(V,x_c)\\
	\mathsc{F-solve}(V)
	&\equiv \exists x \in \mathcal{X} .\, \neg \mathsc{lyap}(V,x)\\
\end{align*}
\begin{tabbing}
    \rule{\linewidth}{0.4pt}\\
    \textbf{algorithm:} Counterexample-guided Lyapunov search\\
    \textbf{initialize}: $k:=0$, $x_0 \in \mathcal{X}$, $\mathcal{X}_0:=\{x_0\}$\\
    \textbf{repeat}: \\
        \quad
        \= 1. \= \emph{Search for a candidate Lyapunov function}.\\
		\>\>\quad\=\textbf{if} $\mathsc{E-solve}(\mathcal{X}_k)$ \textbf{then}: \\
		\>\>\>\quad\= $V_k:=$ solution to $\mathsc{E-solve}(\mathcal{X}_k)$\\
		\>\>\>\textbf{else} \textbf{return} $\mathsc{false/unknown}$\\
		\> 2. \> \emph{Generate counterexample}. \\
		\>\>\>\textbf{if} $\mathsc{F-solve}(V_k)$ \textbf{then}: \\
		\>\>\>\> $x_{k+1}:=$ solution to $\mathsc{F-solve}(V_k)$\\
		\>\>\>\textbf{else} \textbf{return} $\mathsc{true}$\\
		\> 3. \> \emph{Update counterexample set}.\\
		\>\>\>$\mathcal{X}_{k+1}:= \mathcal{X}_k \cup \{x_{k+1}\}$\\
		\>\>\>$k:=k+1$\\
    \textbf{until}: stopping criterion\\
    \rule{\linewidth}{0.4pt}
\end{tabbing}

\section{Extended examples}\label{sec:examples}

\subsection{Verifying a region of attraction}
For an initial application of our counterexample-guided Lyapunov function synthesis
procedure, we consider the linear dynamical system
$x(t+1)=Ax(t)$, where the $2\times 2$ matrix 
\[
	A = 
	\begin{bmatrix}
		0.7005 & -0.2638 \\
		-0.2278 & -0.4627
	\end{bmatrix}
\]
is globally (Schur) stable, with spectral radius $\rho(A)=0.75$. The origin of
the dynamical system is globally exponentially stable.

We take $\mathcal{V}$ to be the class of max-of-affine Lyapunov functions. In
general this class (or the class of \acs{amn}-representable)
Lyapunov function candidates may not be large enough to prove global
exponential stability, even of linear dynamical systems. So instead, we verify
local stability in the box
$\mathcal{B}=\{x\in\reals^2 \mid \|x\|_\infty \leq 10\}$. In other
words, we verify the weaker claim that the set $\mathcal{B}$ is a region of
attraction for the equilibrium at the origin.

Our algorithm terminates with the piecewise affine Lyapunov function
\begin{equation*}
\begin{aligned}
	V(x) &= \max\{
		x_2, 
		-0.1612x_1 -0.1020x_2,
		 0.4614x_1 +0.0155x_2,\\
		&-0.4212x_1 +0.1433x_2,
		-0.5156x_1 +0.0796x_2,
		 0.5036x_1 -0.1632x_2
	\}
\end{aligned}
\end{equation*}
using automatically generated counterexamples that were all constrained to
$\mathcal{B}$, see Figure~\ref{fig:lyap}. Although the function $V(x)$ is a
Lyapunov function by construction, we can illustrate that it decreases along
trajectories of the dynamical system $x(t+1)=Ax(t)$ by simulating the system at
several initial conditions within $\mathcal{B}$, and tracking the value of
$V(x(t))$, see Figure~\ref{fig:trajectories}.

\begin{figure}[htpb]
    \centering
    \begin{subfigure}[b]{0.49\linewidth}
        \centering
		\includegraphics[trim=75 0 50 0,clip,width=\textwidth]{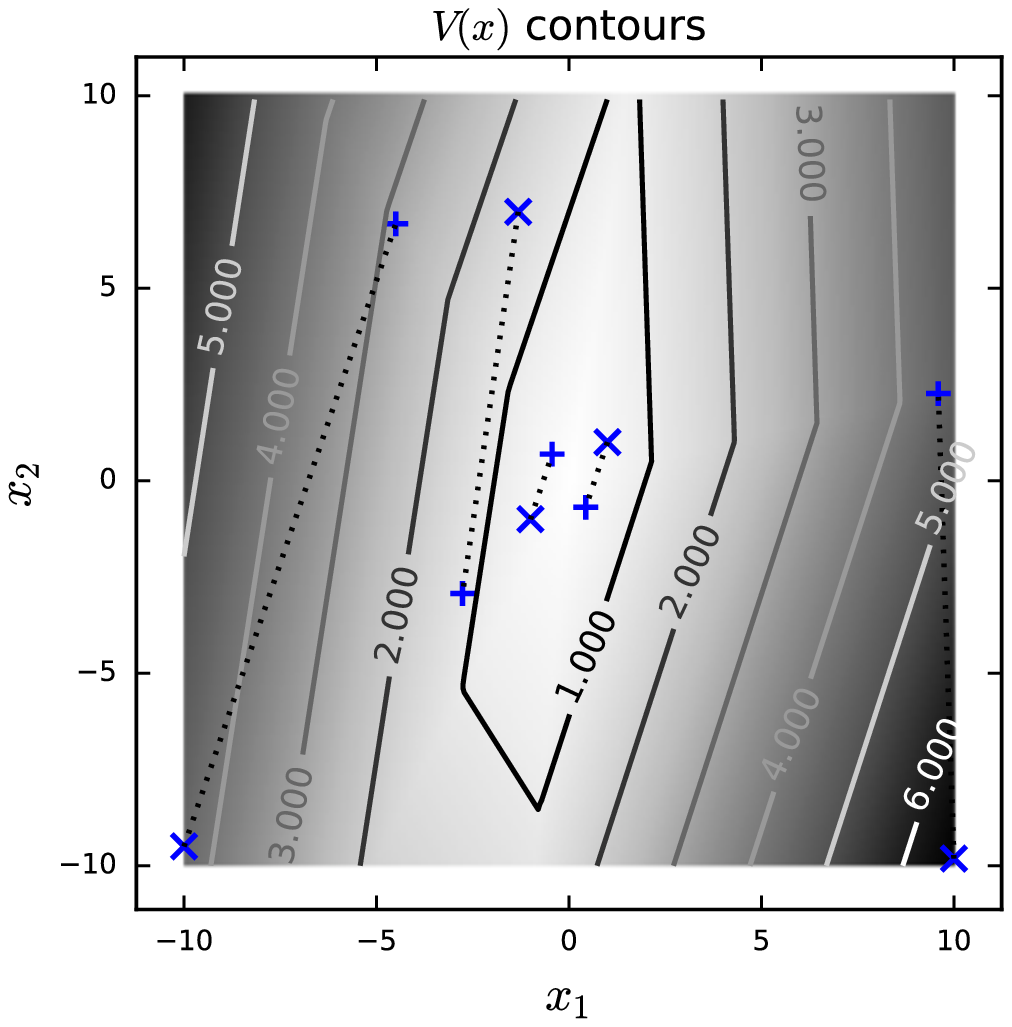}
        \caption{Contours of $V(x)$}
		\label{fig:lyap_contours}
    \end{subfigure}
    \begin{subfigure}[b]{0.49\linewidth}
        \centering
		\includegraphics[trim=90 0 50 0,clip,width=\textwidth]{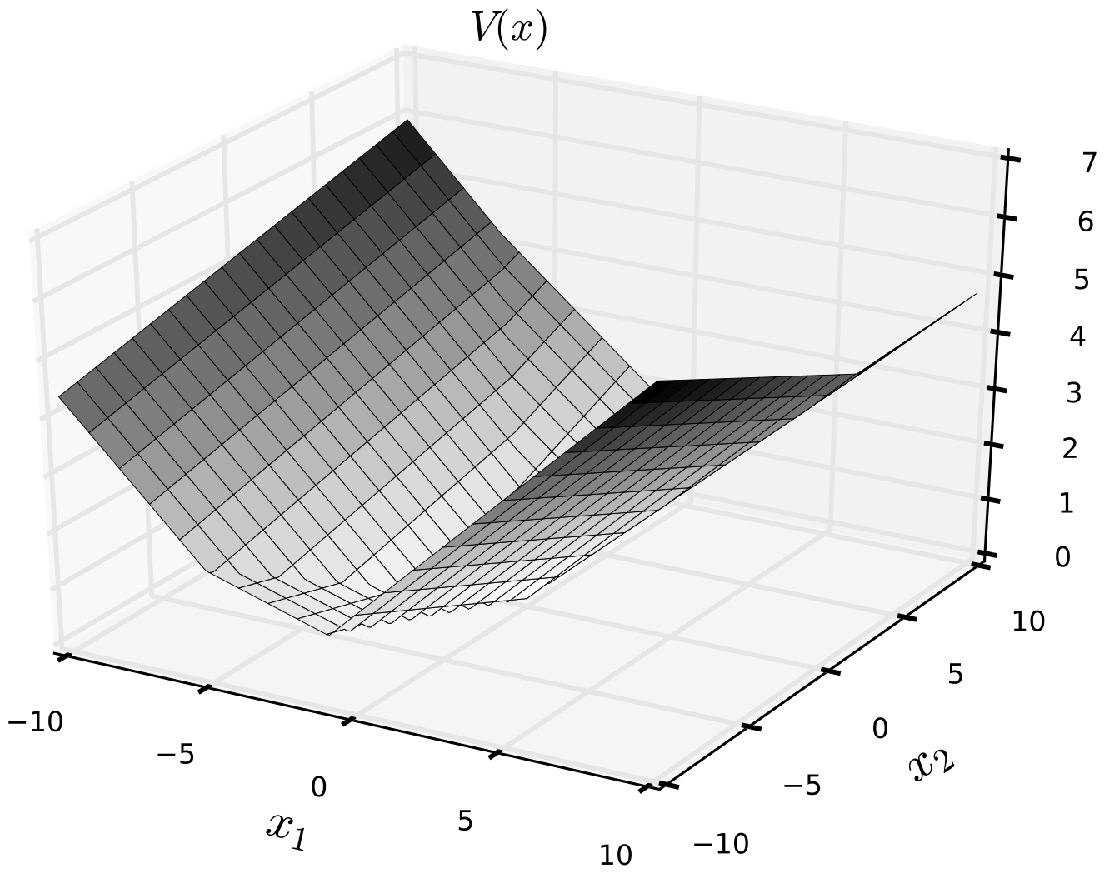}
        \caption{3D visualization of $V(x)$}
		\label{fig:lyap_3d}
    \end{subfigure}
	\caption{Contours (left, solid lines) and 3D visualization (right) of a
	candidate piecewise affine function, synthesized to decrease from the
	initial conditions (counterexamples) marked by a cross, $\times$, to their
	state-space locations marked by a plus, $+$, at the next time step. The
	counterexample conditions were generated automatically using our algorithm.
	The resulting function $V(x)$ certifies asymptotic stability of the origin
	for any initial condition in the box $\mathcal{B}=\{x\in\reals^2 \mid
	\|x\|_\infty \leq 10\}$.}
    \label{fig:lyap}
\end{figure}

\begin{figure}[htpb]
    \centering
    \includegraphics[trim=10 10 20 20,clip,width=0.75\textwidth]{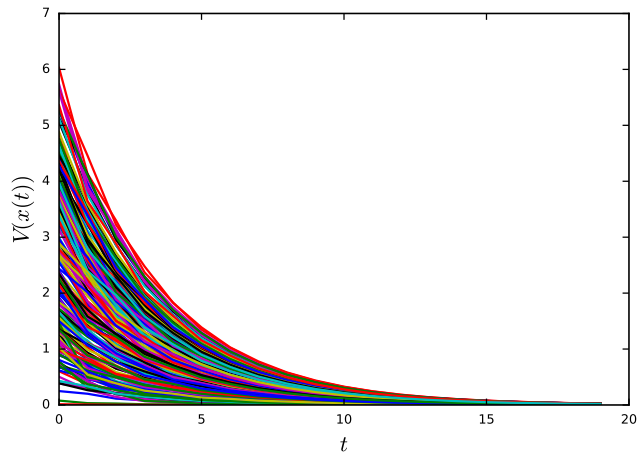}
	\caption{Lyapunov function value for trajectories starting at 200 random
	points in the box $\mathcal{B}=\{x\in\reals^2 \mid \|x\|_\infty \leq 10\}$.}
    \label{fig:trajectories}
\end{figure}

\subsection{Verification of $\lambda$-contractive dynamical systems}
In this example, we verify the $\lambda$-contractiveness of a given closed loop
dynamical system, following the example in~\cite{Milani:1999}.
\begin{definition}[$\lambda$-contractive]
	Let $G$ be a matrix in $\reals^{m\times n}$ and $w$ a vector in $\reals^m$. 
	Define the polyhedron $S(G,w) = \left\lbrace x \mid Gx \preceq w \right\rbrace
	$. The set $S(G,w)$ is \emph{$\lambda$-contractive} with respect to the
	system~\eqref{eq:dtsys} 
	if there is a real $0<\lambda<1$ such that $x(t+1) \in
	S(G,w\epsilon\lambda)$ for all $0< \epsilon \leq 1$ and all $x(t)\in
	S(G,w\epsilon)$, \ie
	\begin{equation*}
		\Phi: \quad 
		\forall x .\, 
		\forall \epsilon .\,
		[(0<\epsilon\leq 1) \wedge (x\in S(G,w\epsilon))]
		\rightarrow (\amn(x)\in S(G,w\lambda \epsilon)).
	\end{equation*}
\end{definition}

Consider the closed loop system with state-feedback control
represented by the state equations
\begin{equation}
	\begin{aligned}
		x(t+1) &= Ax(t) + Bu(t), \\
		u(t) &= \text{sat}(Fx(t)),
	\end{aligned}
	\label{eq:SysEq} 
\end{equation}
where $x(t) \in \reals^n$, $u \in \reals$, $A \in \reals^{n\times n}$, $B \in
\reals^{n}$, and
\begin{align*}
\text{sat}(Fx) &= \begin{cases} -u^\text{min} &\text{if} ~ Fx < -u^\text{min}, 
\\Fx &\text{if} ~ -u^\text{min} \leq Fx \leq u^\text{max},
\\u^\text{max} &\text{if} ~ Fx > u^\text{max}.
\end{cases}
\end{align*}
In other words, the autonomous system has $\amn(x(t)) = Ax(t) + B\sat(Fx(t))$.

Given $G\in \reals^{r\times n}$ and $w \succ 0 \in \reals^r$, we can test the
$\lambda$-contractiveness of $S(G,w)$ by defining two AMN functions $V_1,V_2
: \reals^n \rightarrow \reals$, whose zero-sublevel sets represent 
the region inside the polyhedron $S(G,w)$.  We rearrange the inequality
$Gx\preceq w$, where $g_i$ is the $i$-th row of $G$, to give
\begin{align*}
	V_1(x) &= \max_i (g_ix - \epsilon w_i),\\
	V_2(\amn(x)) &= \max_i (g_i \amn(x) - \epsilon \lambda w_i).
\end{align*}

The negative of the condition $\Phi$ is
\begin{align*}
\label{eq:SatSpec} \neg \Phi &= \exists x,\epsilon.\, 
	x\in S(G,w\epsilon)\wedge 0<\epsilon\leq 1
	\wedge \neg(\amn(x) \in S(G,w\epsilon \lambda) ) \\
	&= \exists x,\epsilon.\,  (V_1(x)\leq 0) \wedge (V_2(x) > 0 ) \wedge (0< \epsilon \leq 1)
\end{align*}
If there exist $x$ and $\epsilon$ that satisfy $\neg\Phi$, then the
polyhedron $S(G,w)$ is not $\lambda$-contractive with respect to
system~\eqref{eq:SysEq}. However, if $\neg \Phi$ is \textsc{unsat}, then we can
say that the polyhedron $S(G,w)$ is $\lambda$-contractive within the region of
nonlinear behavior of system~\eqref{eq:SysEq}.

The example polyhedron $S(G,w)$ in~\cite{Milani:1999} can be verified as
$\lambda$-contractive:
\[
	G = \begin{bmatrix}
	0.2888 & -1.8350 \\
	0.9650 & -2.0576 \\
	1.0008 & 1.7891 \\
	1.5951 & -1.9866 \\
	2.0707 & -2.0590 \\
	-1.4970 & -1.5864 \\
	-0.2888 & 1.8350 \\
	-0.9650 & 2.0576 \\
	-1.0008 & -1.7891 \\
	- 1.5951 & 1.9866 \\
	1.4970 & 2.0590 \\
	- 2.0707 & 1.5864
	\end{bmatrix},
	\quad
	w = \begin{bmatrix}
	35.4375 \\
	48.2116 \\
	48.1152 \\
	62.5184 \\
	62.3934 \\
	76.2996 \\ 
	35.4375 \\
	48.2116 \\
	48.1152 \\
	62.5184 \\
	62.3934 \\
	76.2996 
	\end{bmatrix}.
\]
However, if we scale $w$ by $\delta = 1.01$, then we can find a counterexample,
\[
	x(t) = \begin{bmatrix} 
	38.9278177 \\ 2.27698913
	\end{bmatrix},
	\quad
	x(t+1) = \begin{bmatrix} 32.28074873 \\ -5.83874012	\end{bmatrix},
	\quad
	\epsilon = 0.99914198.
\]

The reference~\cite{Milani:1999} provides a set of independent linear programs,
a solution of which would synthesize the polyhedron $S(G,w)$, which is
$\lambda$-contractive with respect to closed loop system.  With AMNs, we have
an alternate, fully automatic method to validate piecewise affine Lyapunov
functions for stability analysis of LTI discrete time systems with saturated
closed loop control inputs.

\subsection{Nonconvex optimal control}
This example analyzes finite horizon optimal control problems with affine (both
convex and nonconvex) control constraints.  Consider the piecewise affine
discrete linear system
\[
	x(t+1) = Ax(t) + Bu(t),
\]
where
\[
	A = 
	\begin{bmatrix} 
			1 & \delta t \\ 
			0 & 1 
	\end{bmatrix},
	\quad
	B = 
	\begin{bmatrix}
		\frac{\delta t^2}{2m} \\
		\frac{\delta t}{m} 
	\end{bmatrix}, 
\]
and initial and goal states
\[
	\begin{bmatrix} 
		x_1(0)\\
		x_2(0)
	\end{bmatrix} = 
	\begin{bmatrix} 
		0\\
		0
	\end{bmatrix},
	\quad
	\begin{bmatrix} 
		x_1(N)\\
		x_2(N)
	\end{bmatrix} = 
	\begin{bmatrix} 
		1\\
		0
	\end{bmatrix}.
\]
The control magnitude is bounded above and below by
\[
	0.2/T^\text{tot} \leq \|u(t)\|_1 \leq 1/T^\text{tot}, \quad 
	T^\text{tot}=7.5\text{s}.
\]
Note that the upper bound is a convex constraint, while the lower bound is not.
The goal is to reach the goal state while minimizing the objective
\[
	J = \sum_{t=0}^{N-1} \|u(t)\|_1.
\]
We solve the (nonconvex) optimization problem
\[
	\begin{array}{ll}
		\mbox{minimize}   & \sum_{t=0}^{N-1} \|u(t)\|_1 \\
		\mbox{subject to} 
			& x(t+1) = A x(t) + Bu(t), \quad t=0,\ldots,N-1\\
			& 0.2/T^\text{tot} \leq \|u(t)\| \leq 1/T^\text{tot}, \quad t=0,\ldots,N\\
			& x(0) = (0, 0), \quad x(N) = (1, 0).
	\end{array}
\]
with \amnet. Figure~\ref{fig:noncvx_control} summarizes the results. Note that
the lower bounds are enforced for the nonconvex optimization problem, leading
to a control schedule that obeys the lower bound on the control magnitude by
never coasting with $u(t)=0$; such a control schedule is not obvious from the
optimal schedule for the convex problem, but it is provably optimal.

\begin{figure}[tbp]
	\centering
	\begin{subfigure}{0.95\textwidth}
		\includegraphics[width=\textwidth]{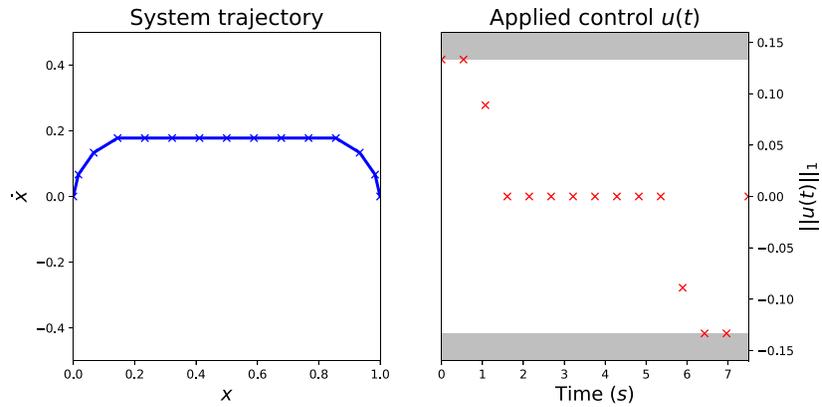}
		\caption{Convex (\textsc{cvxpy})}
		\label{fig:noncvx_control_a}
	\end{subfigure}\\
	\begin{subfigure}{0.95\textwidth}
		\includegraphics[width=\textwidth]{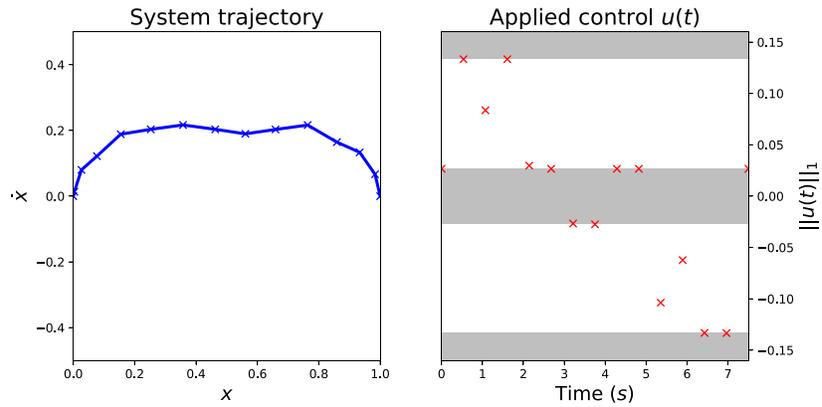}
		\caption{Combined (\amnet)}
		\label{fig:noncvx_control_b}
	\end{subfigure}
	\caption{Illustration of nonconvex optimization capabilities of \amnet.
	Optimal control trajectories for the convex problem (upper control bound
	only) using \textsc{cvxpy}~\cite{cvxpy} 
	and  combined (upper and lower control bounds) using \amnet.
	Shading indicates disallowed control inputs.} 
	\label{fig:noncvx_control}
\end{figure}

\clearpage

\subsection{Characterizing classifier robustness}
To illustrate the use of \acsp{amn} in other neural network applications, we
used \textsc{TensorFlow} to train a simple and small neural network classifier
on the popular MNIST handwritten digit
dataset~\cite{TensorFlow:2015,LeCun:1998}.  To make conversion and verification
of the corresponding \acs{amn}
manageable, the dimension of the training data was first reduced from 784 
($28\times28$ handwritten digit images) to 40 using PCA whitening.
The resulting classifier was automatically converted to an \acs{amn}, and its
associated input-output relationship encoded as SMT constraints with our
accompanying \textsc{Amnet} modeling
toolbox.
The baseline, 20-unit hidden layer \acs{amn} 
$\amn^\mathsc{mnist}:\reals^{40}\to\reals^{10}$,
with ReLU nonlinearities between the
layers, achieved a classification rate of 
0.8736
on the test data set.

\begin{figure}[htbp]
    \centering
    \includegraphics{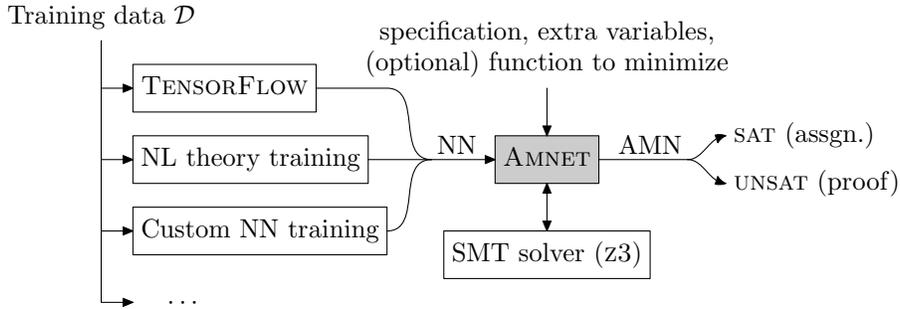}
    \caption{Using \textsc{Amnet} to convert a NN to an~\acs{amn}.}
    \label{fig:amnet1}
\end{figure}

The procedure to convert a NN with piecewise affine nonlinearities to an
\acs{amn} is entirely mechanical (\S{}\ref{sec:encoding}), and importantly,
indifferent to the specific optimization algorithm (\eg, gradient descent,
batching, regularization) used to train the original classifier. Therefore,
once the weights and biases of the NN are encoded in SMT, formal properties
of the original NN can be readily verified using our toolbox, see
Figure~\ref{fig:amnet1}.

A perturbation $\varepsilon = (\varepsilon_1, \ldots,
\varepsilon_5, 0, \ldots, 0) \in \reals^{40}$ was created to act on the 5 most significant
components of the dimensionality reduced input, $X_\mathrm{amn} =
X_\mathrm{pca} + \varepsilon$.
Each dimension of the perturbation was constrained to $-3 \leq \varepsilon_i
\leq 3$.  
The output layer of $\amn$ was constrained to produce a misclassification of
`5' by introducing the equality constraint
\[
\max(\amn^\mathsc{mnist}(X_\mathrm{amn})) =
\amn^\mathsc{mnist}(X_\mathrm{amn})_6.
\]
on the final classification layer.
The \textsc{z3} SMT solver was used to find a solution to these constraints,
shown below.  The resulting perturbed image was recovered by $X^T =
V^{T}X_\mathrm{amn}$.

\begin{figure}[htbp]
	\begin{subfigure}{.30\textwidth}
		\centering
		\includegraphics[width=\textwidth]{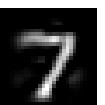}
		\caption{Original image; classifies as `7'.}
		\label{fig:HOG_resolution}
	\end{subfigure}
	~
	\begin{subfigure}{.30\textwidth}
		\centering
		\includegraphics[width=\textwidth]{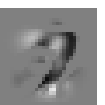}
		\caption{Perturbation visualization.}
		\label{fig:block_overlap}
	\end{subfigure}
	~
	\begin{subfigure}{.30\textwidth}
		\centering
		\includegraphics[width=\textwidth]{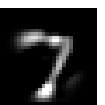}
		\caption{Perturbed image; classifies as `5'.}
		\label{fig:Peturbed_image}
	\end{subfigure}
	\caption{Perturbation on MNIST images.}
	\label{fig:MNIST_result}
\end{figure}

\section{Conclusion and extensions}\label{sec:extensions}
In this paper, we introduced the concept of an affine multiplexing network
(AMN), which is a relation built up using two fundamental building blocks: a
multiplexing function ($\mu$), and an affine transformation ($\alpha$). By
restricting to these two building blocks, it is possible to formally encode the
relation in linear arithmetic, and therefore operate on it with existing SMT
and MIP solvers. We applied this framework to nonlinear controller
verification, and synthesis of stability proofs of neural networks in the loop.
We also introduced the software package \amnet\ to make modeling with AMNs
simple.  Many extensions are possible, including the following subjects of
future work:
\begin{itemize}
	\item \textbf{Conic enable condition} Instead of $z \leq 0$, we can take
	$z\in\mathcal{K}$ to be an arbitrary cone. The multiplexer nonlinearity
	becomes 
	\[
		\mu_\mathcal{K}(x,y,z)
		=
		\left\{
		\begin{array}{ll}
			x, & \text{if } z \in \mathcal{K},\\
			y, & \text{otherwise}.
		\end{array}
		\right.
	\]
	The standard nonlinearity is the same as $\mu(x,y,z) =
	\mu_{-\reals}(x,y,z)$. The verification problem translates to
	conic existential conditions, which can be tractably computed with a
	convex modification to the DPLL algorithm in SMT solvers.
	\item \textbf{Simulating simple programs} Many programs can be written as if-then-else networks. In
	general, as long as a closed loop system can be written as
	$x(t+1)=\amn(x(t))$, then AMNs can be used to model them. 
	\item \textbf{Continuous time dynamics} Dynamics governed by differential equations
    $\dot{x} = \amn(x)$ can be treated just as well as discrete time dynamics governed 
    by difference equations $x(t+1)=\amn(x(t))$, under regularity assumptions the 
    function $\amn$, which guarantee existence and uniqueness of solutions.
	\item \textbf{Path planning and Model Predictive Control} Since we can represent constraints
	$x(t+1) = \amn(x(t))$, $t=0,\ldots, N-1$  over a time horizon $N$, we can
	do optimal path planning for any (potentially nonlinear) system, including
	switched systems. This can help us answer questions about whether a
	particular policy for a small system (\eg, a switched power converter) is
	optimal or only near optimal, and to come up with novel solutions for
	nonlinear systems. 
	\item \textbf{Vision in the loop}
	As long as every component of a closed loop system is (or can be modeled
	as) an AMN, formal verification of vision in the loop systems can in
	principle be attempted. See Figure~\ref{fig:nnvis1}.
		\begin{figure}[htpb]
			\centering
			\includegraphics{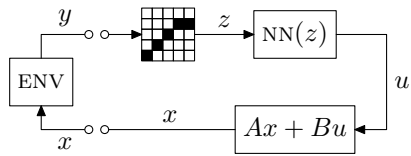}
			\caption{Vision system in the loop with AMN environment model
			\textsc{env}.}
			\label{fig:nnvis1}
		\end{figure}
	\item \textbf{Convolutional linear units} Convolutional units are simply
	affine units with a special case Toeplitz structure, which admit
	specialized algorithms. We can take advantage of special structure forming
	convolution equality constraints like $y = c * x$, where $c$ is the
	convolution kernel.
\end{itemize}

\section{Acknowledgments}
We wish to thank R. Dimitrova, M. Ahmadi, and H. Poonawala for insightful
discussions, and to acknowledge the grants 
AFRL FA8650-15-C-2546 and
AFRL UTC 17-S8401-10-C1.
This work was completed while the first author was an Institute for
Computational Engineering and Sciences (ICES) fellow at the University of Texas
at Austin.



\nocite{DeMoura:2008}  
\nocite{Vavasis:2010,Nemirovskii:1993} 
\nocite{Blondel:2001}  
\nocite{Kapinski:2014} 
\nocite{Ravanbakhsh:2015}	
\nocite{Hassibi:1998} 
\nocite{Goncalves:2001} 
\nocite{Kroening:2016} 
\nocite{Huang:2017}   
\nocite{Cheng:2017}   
\nocite{Katz:2017}    
\nocite{Pulina:2010}  
\nocite{Bastani:2016} 
\nocite{Hunnekens:2016}
\nocite{Foerster:2017} 
\nocite{Sontag:1992}   
\nocite{Iandola:2016}  
\nocite{Giesl:2010}    
\nocite{Zhu:2017}      
\nocite{Dvijotham:2018, Bunel:2017} 
\nocite{Dutta:2017}    

\bibliographystyle{alinit}
\bibliography{amn}
\end{document}